\documentclass[a4paper,11pt]{amsart}
\usepackage{amsmath,amssymb,amsfonts,amsthm,exscale,calc}
\usepackage{latexsym,textcomp}
\usepackage{graphicx,epsf}
\usepackage{a4wide}
\usepackage{dsfont}
\usepackage[german, english]{babel}
\usepackage{hyperref}
\usepackage{enumerate}
\usepackage{tikz}
\usepackage[europeanresistors,americaninductors]{circuitikz}
\usepackage[utf8]{inputenc}

\newtheorem{lemma}{Lemma}[section]
\newtheorem{theorem}[lemma]{Theorem}

\newtheorem{proposition}[lemma]{Proposition}
\newtheorem{corollary}[lemma]{Corollary}

\theoremstyle{definition}

\newtheorem{example}[lemma]{Example}
\newtheorem*{remark}{Remark}

\numberwithin{equation}{section}
\newcommand{\comment}[1]{}

\newcommand{\Wm}{{W_{\rm min}}}
\newcommand{\Ws}{{W_{\rm max}}}

\newcommand{\R}{{\mathbb R}}
\newcommand{\IC}{{\mathbb C}}

\newcommand{\N}{{\mathbb N}}

\newcommand{\LL}{{\mathcal{L}}}
\newcommand{\MM}{{\mathcal{M}}}

\newcommand{\supp}{{\mathrm {supp}\,}}

\newcommand{\as}[1]{\langle #1\rangle}

\newcommand{\Hm}[1]{\leavevmode{\marginpar{\tiny%
$\hbox to 0mm{\hspace*{-0.5mm}$\leftarrow$\hss}%
\vcenter{\vrule depth 0.1mm height 0.1mm width \the\marginparwidth}%
\hbox to 0mm{\hss$\rightarrow$\hspace*{-0.5mm}}$\\\relax\raggedright
#1}}}

\begin{document}

\title[Surjectivity of discrete operators]{A note on the surjectivity of operators on vector bundles over discrete spaces}

\author[Koberstein]{Jannis Koberstein}
\address{J. Koberstein, Mathematisches Institut \\Friedrich Schiller Universität Jena \\07743 Jena, Germany }\email{jannis.koberstein@uni-jena.de}

\author[Schmidt]{Marcel Schmidt}
\address{M. Schmidt, Mathematisches Institut \\Friedrich Schiller Universität Jena \\07743 Jena, Germany } \email{schmidt.marcel@uni-jena.de}

\begin{abstract}
In this note we give a short and self-contained proof for a  criterion of Eidelheit on the solvability of linear equations in infinitely many variables. We use this criterion to study the surjectivity of magnetic Schrödinger operators on bundles over graphs. 
\end{abstract}


\maketitle



\section{Introduction}

Given a double sequence  $(a_{nm})$  of complex numbers (or of linear operators) it is a fundamental question for which $(\eta_n)$ there exists  $(\xi_n)$ that solves the infinite system of linear equations
$$\sum_{m=1}^\infty a_{nm} \xi_m = \eta_n, \, n \in \N.$$
For continuous linear functionals  $(a_{nm})$  on Fr\'echet spaces a classical result of Eidelheit \cite{Eid} characterizes the solvability of this system of equations for all $(\eta_n)$. In other words it provides a characterization for the surjectivity of the map 
$$(\xi_n) \mapsto (\sum_{m=1}^\infty a_{nm} \xi_m).$$
The proof of Eidelheit's result is based on a general surjectivity criterion for linear operators between Fr\'echet spaces due to Orlicz and Mazur \cite{MO1,MO2}. We refer to \cite{MV} for a presentation of this classical result with modern notation and terminology, which basically follows Eidelheit's original arguments. 

This note has two purposes: We  give a short functional analytic proof for a special case of Eidelheit's theorem, which does not use the previously mentioned surjectivity criterion for operators on Fr\'echet spaces. More precisely, we consider the situation when $(a_{nm})$ is a sequence of linear operators between finite dimensional vector spaces with finite hopping range (that is for each $n\in \N$ we assume $a_{nm} = 0$ for all but finitely many $m$), see Theorem~\ref{theorem:main}. Secondly, we apply the theorem to obtain a criterion for the surjectivity of magnetic Schrödinger operators on bundles over   graphs, see Theorem~\ref{thm:graph}.  A particular consequence of our presentation is a short and self-contained  proof for the surjectivity of the (weighted) graph Laplacian of an infinite connected locally finite graph, which uses the bipolar theorem as the only nontrivial ingredient from functional analysis. 

The surjectivity of graph Laplacians has recently received some attention. Based on surjectivity results for linear cellular automata from \cite{CC2}, the surjectivity of the graph Laplacian was first established for Cayley graphs of finitely generated infinite amenable groups in \cite{CC1} and then extended to arbitrary connected locally finite infinite graphs in \cite{CCD}.    The proofs in these papers are based on a Mittag-Leffler argument, which is well known to have applications in all kinds of surjectivity problems, see e.g. \cite{Wen} and references therein. With the same arguments a version of Eidelheit's theorem for linear cellular automata was proven in \cite{Toi}. There it is called a 'Garden of Eden theorem'\footnote{This comes from the observation that initial states which are not in the image of a given cellular automaton can never be attained after iterating it.   A garden of Eden theorem is then a theorem that provides criteria for the non-existence of such states, that is, criteria for the surjectivity of the automaton.}. It seems that the authors of the previously mentioned papers were not aware of Eidelheit's theorem and  it was only noted in \cite{Kal} that Eidelheit's theorem can be applied to obtain the surjectivity of graph Laplacians and other discrete operators that satisfy a pointwise maximum principle. 

As mentioned above, we prove Eidelheit's surjectivity criterion for finite hopping range operators on finite-dimensional vector bundles over infinite discrete set.  It includes the result for linear cellular automata from \cite{Toi} but can also be applied to magnetic Schrödinger operators on bundles, which were recently introduced in \cite{GMT}. Our result for magnetic Schrödinger operators (Theorem~\ref{thm:graph})  contains the aforementioned results for graph Laplacians and discrete Schrödinger operators (graph Laplacian plus real potential). Instead of a pointwise maximum principle as in \cite{Kal} we assume the nonnegativity of a certain quadratic form associated with the magnetic Schrödinger operator. For scalar discrete Schrödinger operators we give a characterization of the pointwise maximum principle  and, in doing so,  show that our result covers situations where the maximum principle fails. 

This paper is organized as follows. In Section~\ref{section:main results} we introduce the model and discuss the main results. In Section~\ref{section:background} we proof Theorem~\ref{theorem:main} and in Section~\ref{section:applications} we prove Theorem~\ref{thm:graph} and discuss the maximum principle. In Section~\ref{section:examples}  we give several examples that show that we cannot drop any of the assumptions in Theorem~\ref{thm:graph}.

Parts of this paper are based on the first named authors master's thesis.

{\bf Acknowledgements:} The authors are grateful to thank Daniel Lenz for pointing out the problem to them. Moreover, M.S. thanks Jürgen Voigt for an interesting discussion on closed range theorems.

\section{Setup and main results}\label{section:main results}

Let $X\neq \emptyset$ be a countable set. A {\em vector bundle} over $X$ is a family $F = (F_x)_{x \in X}$ of finite-dimensional complex vector spaces. The corresponding {\em space of vector fields} is  
$$\Gamma(X;F) := \prod_{x\in X} F_x = \{f\colon X \to \bigsqcup_{x\in X} F_x \mid f(x) \in F_x \text{ for all }x \in X\}$$
and the subspace of {\em finitely supported vector fields} is given by
$$\Gamma_c(X;F) := \{f \in \Gamma(X;F) \mid \text{ there exists a finite } K\subseteq X \text{ such that } f|_{X \setminus K} = 0\}.$$
For each $x \in X$ we equip $F_x$ with the unique vector space topology on it and for simplicity we fix a complete norm $\|\cdot\|_x$ that induces this topology.  We equip $\Gamma(X;F)$ with the product topology of the family $F = (F_x)_{x \in X}$. It coincides with the locally convex topology generated by the family of seminorms
$$p_K\colon\Gamma(X; F) \to [0,\infty),\quad p_K(f) = \sum_{x \in K} \|f(x)\|_x,\quad K \subseteq X \text{ finite}.$$
It is readily verified that the continuous dual space of $\Gamma(X;F)$ is isomorphic to $\Gamma_c(X;F')$, where $F' = (F_x')_{x \in X}$ is the {\em dual vector bundle} that consists of the dual spaces of the $F_x$, see also Section~\ref{section:background}.

A linear operator $A\colon\Gamma(X;F) \to \Gamma(X;F)$ is continuous if and only if for each $x \in X$ there exists a finite subset $K \subseteq X$ such that $Af(x)$ only depends on $f|_K$. Operators with the latter properties are said to have {\em finite hopping range}. In this case, its {\em dual operator}  $A'\colon\Gamma_c(X;F') \to \Gamma_c(X;F')$ is defined by
$$A' \varphi (f) = \varphi(Af),\quad f \in \Gamma(X;F), \varphi \in\Gamma_c(X;F'). $$

As mentioned above, it is one of the goals of this paper to give a short and self-contained functional analytic proof for the following result. It is a generalization of the corresponding result for linear maps between finite-dimensional spaces.

\begin{theorem}\label{theorem:main}
 Let $A\colon\Gamma(X;F) \to \Gamma(X;F)$ be a continuous linear operator. The following assertions are equivalent.
 \begin{itemize}
  \item[(i)] $A$ is surjective.
  \item[(ii)] The dual operator  $A'\colon\Gamma_c(X;F') \to \Gamma_c(X;F')$ is injective.
 \end{itemize}
\end{theorem}

\begin{remark}
Since the fibers of the bundle $F$ are finite-dimensional vector spaces and $X$ is countable, $\Gamma(X;F)$ is isomorphic to the space of all complex valued sequences. For continuous operators (finite hopping range operators) on this space  it was noted in \cite{Kal} that the previous theorem can be deduced from Eidelheit's theorem. The case when all fibers of $F$ are equal is treated in \cite{Toi} with a Mittag-Leffler argument.
\end{remark}

The second goal of this paper is to apply the previous theorem to graph Laplacians and, more generally, discrete magnetic Schrödinger operators. This is discussed next. 

A {\em weighted graph} is a pair $(X,b)$, where $X$ is a countable set and $b\colon X\times X \to [0,\infty)$ is a map with the properties
\begin{enumerate}
\item[(b0)] $b(x,x) = 0$ for all $x\in X$,
 \item[(b1)] $b(x,y) = b(y,x)$ for all $x,y \in X$, 
 \item[(b2)] ${\rm deg}(x) := \sum_{y \in X}b(x,y) < \infty$ for all $x \in X$.
\end{enumerate}
The elements of $X$ are then interpreted as {\em vertices} of a graph; two such vertices $x,y \in X$ are connected by an {\em edge} if $b(x,y) > 0$, in which case we write $x \sim y$. If, additionally, for every $x \in X$ the set of {\em neighbors of $x$}
$$\{y \in V \mid b(x,y) >0\}$$
is finite, then $b$ is called {\em locally finite}. For $n \in \N$ a {\em path} of {\em length $(n-1)$}  is a finite sequence of vertices $x_1,\ldots, x_n$ such that for all $i = 1,\ldots,n-1$ we have $x_i \sim x_{i+1}$. Two vertices $x,y \in X$ are said to be {\em connected by a path} if they are contained in a path. In this case, {\em the combinatorial distance $d(x,y)$} between $x$ and $y$ is the length of the shortest path containing $x$ and $y$. Being connected by a path is an equivalence relation on $X$ and its equivalence classes are called {\em connected components}. 

We write $C(X)$ for the complex-valued functions on $X$ and $C_c(X)$ for the subspace of functions of finite support. For a given graph $(X,b)$ and a {\em potential} $V\colon X \to \mathbb R$ we define the quadratic form $q_V := q_{b,V}$ by
$$q_V\colon C_c(X) \to \mathbb R, \quad q_V(f) = \frac{1}{2} \sum_{x,y \in X} b(x,y)|f(x) - f(y)|^2 +  \sum_{x \in X} |f(x)|^2 V(x).$$
We say that $q_V$ is {\em nonnegative} if $q_V(f) \geq 0$ for all $f \in C_c(X)$.

We call a vector bundle $F = (F_x)_{x \in X}$ {\em Hermitian} if for all $x \in X$ the space $F_x$ is equipped with a complete inner product $\as{\cdot,\cdot}_x$. In this case, a {\em connection} on $F$ is a family of unitary maps $ \Phi = (\Phi_{xy} \colon F_y \to F_x)_{x,y \in X}$ with the property that $\Phi_{yx} = \Phi_{xy}^{-1}$ for all $x,y \in X$. A {\em self-adjoint bundle endomorphism} on $F$ is  a family of self-adjoint linear maps $W= (W(x)\colon F_x \to F_x)_{x\in X}$ and for $x \in X$ we denote by $\Wm(x)$ the smallest eigenvalue of $W(x)$ and by  $\Ws(x)$ the largest eigenvalue of $W(x)$. 

Given a graph $(X,b)$, a Hermitian vector bundle $F$, a connection $\Phi$  and a bundle endomorphism  $W$, the domain of the associated {\em magnetic Schr\"odinger operator}  $\MM = \MM_{b,\Phi,W}$ on $\Gamma(X;F)$ is 
$$D(\MM) := \{f \in \Gamma(X;F) \mid \sum_{y\in X} b(x,y) \|f(y)\|_y < \infty \text{ for all }x \in X\}, $$
on which it acts by
$$\MM f(x) := \sum_{y \in X} b(x,y) (f(x) - \Phi_{xy} f(y)) + W(x) f(x).$$
If $(X,b)$ is locally finite, then $D(\MM) = \Gamma(X;F)$ and $\MM$ is continuous as it has finite hopping range. 

In the case when $F = (\IC) $ (the trivial line bundle) and $(\Phi_{xy}) = ({\rm Id}_\IC)$, we can identify $\Gamma(X;F)$ with $C(X)$ and $\Gamma_c(X;F')$ with $C_c(X)$ and any self-adjoint bundle endomorphism on $\Gamma(X;F)$ acts as multiplication on $C(X)$ by a potential $V\colon X \to \mathbb R$. In this case, the operator $\MM$ is  the   {\em graph Laplacian} (plus potential $V$), which we denote by $\mathcal L = \mathcal{L}_V$. It acts on its domain
$$D(\mathcal L) = \{f \in C(X) \mid \sum_{y \in X}b(x,y)|f(y)| < \infty \text{ for all } x \in X\} $$
by
$$\mathcal L f(x)  = \sum_{y \in X}b(x,y)(f(x) - f(y)) + V(x)f(x).$$
For magnetic Schrödinger operators on graphs our main theorem reads as follows.

\begin{theorem}\label{thm:graph}
 Let $(X,b)$ be a weighted graph, let $F$ be a Hermitian vector bundle over $X$ and let $W\colon F \to F$ be a self-adjoint bundle endomorphism. Furthermore, assume the following three conditions.
 \begin{enumerate} 
  \item $(X,b)$ is locally finite.
  \item All connected components of $(X,b)$ are infinite. 
  \item  The quadratic form $q_\Wm$ is nonnegative or the quadratic form $q_{-\Ws - 2{\rm deg}}$ is nonnegative. 
 \end{enumerate}
Then the operator $\MM$ is surjective.
\end{theorem}

\begin{remark}
\begin{enumerate}[(a)]
\item  For the graph Laplacian without potential this result is contained in \cite{CCD}. Moreover, \cite{Kal} contains the surjectivity of $\LL_V$ (the scalar case) under the condition that $\LL_V$ satisfies a pointwise maximum principle. This maximum principle is equivalent to the fact that at any $x \in X$ we have either $V(x) \geq 0$ or $V(x) + 2 {\rm deg}(x) \leq 0$, see Proposition~\ref{proposition:maximum principle}.  The nonnegativity of $q_V$ is of course satisfied if $V \geq 0$ but this is not necessary. If $q_{0}$ satisfies some Hardy inequality, see e.g. the discussion in \cite{KPP2,KPP1}, or the graph has positive Cheeger constant, see e.g. \cite{BKW}, then also certain $V$ without a fixed sign induce a nonnegative form $q_V$. In this case, we can choose $V$ without fixed sign that also satisfies $V + 2{\rm deg} > 0$.  This shows that for Schrödinger operators our theorem can treat potentials that violate the pointwise maximum principle  of \cite{Kal}. 
 \item The local finiteness condition ensures the continuity of the operator $\MM$ while the other two conditions guarantee that its dual is injective. We shall  see that the previous theorem is optimal in the sense that there are counterexamples when dropping any of the assumptions. However, we also give an example  that shows that neither the nonnegativity of $q_V$ nor of $q_{-V - 2{\rm deg}}$ is necessary for the surjectivity of $\LL_V$, see Section~\ref{section:examples}.
\end{enumerate}

\end{remark}

\section{Proof of Theorem~\ref{theorem:main}}\label{section:background} Let us sketch the proof of Theorem~\ref{theorem:main}. In general it is a consequence of the Hahn-Banach theorem that the surjectivity of a continuous linear operator between locally convex vector spaces implies that its dual operator is injective. Since the dual space of $\Gamma(X;F)$ is given by $\Gamma_c(X;F')$, in our situation this implication can be verified explicitly without using the Hahn-Banach theorem. Hence, it suffices to prove that the injectivity of $A'$ implies the surjectivity of $A$.  For $K \subseteq X$ finite we set 
$$U_K := \{f \in \Gamma(X;F) \mid p_K(f) \leq 1\}.$$
We will show that if $A'$ is injective, then for every finite $\emptyset \neq K \subseteq X$ there exists $\varepsilon > 0$ and finite $ \emptyset \neq K'\subseteq X$  such that 
 $$\overline{AU_K} \supseteq \varepsilon U_{K'}.$$
With a standard argument we remove the closure in the above inclusion such that for some $0< \varepsilon' < \varepsilon$ we have $AU_K \supseteq \varepsilon' U_{K'}$ . From this we deduce surjectivity via 
$$A\Gamma(X;F) = \bigcup_{n \in \N} n AU_K  \supseteq \bigcup_{n \in \N} n\varepsilon' U_{K'} = \Gamma(X;F).$$

Before giving the details we recall some elementary facts about polar sets. First we note that an isomorphism between the continuous dual of $\Gamma(X;F)$ and $\Gamma_c(X;F')$ is given as follows. We denote by $(\cdot,\cdot)_x\colon F_x' \times F_x \to \mathbb C$ the dual pairing between $F_x'$ and $F_x$, that is $(\varphi,f)_x := \varphi(f)$ for $\varphi \in F'_x$ and $f \in F_x$. Then the map
$$\Gamma_c(X;F') \to \Gamma(X;F)',\quad \varphi \mapsto (\varphi,\cdot)$$
with
$$(\varphi,f) = \sum_{x \in X} (\varphi(x),f(x))_x, \quad f \in \Gamma(X;F),$$
is a vector space isomorphism and we tacitly identify $\Gamma_c(X;F')$ with $\Gamma(X;F)'$ via this map.
%
%
%
%
%

Recall that the polar sets of $M \subseteq \Gamma(X;F)$ and of $N \subseteq \Gamma(X;F)' = \Gamma_c(X;F')$ are defined by 
$$M^\circ = \{\varphi \in \Gamma_c(X;F') \mid |(\varphi,f)| \leq 1 \text{ for all } f \in M \}$$
and 
$$N^\circ = \{f \in \Gamma(X;F) \mid |(\varphi,f)| \leq 1 \text{ for all } \varphi \in N \}.$$
The bipolar theorem  states that for any convex $M \subseteq \Gamma(X;F)$ we have
$$\overline{M} = (M^{\circ})^{\circ},$$
where $\overline{M}$ denotes the closure of $M$ in $\Gamma(X;F)$. We can now prove the main lemma for the surjectivity of $A$.
\begin{lemma}\label{lemma:inclusion}
 Let $A'$ be injective. For every finite $\emptyset \neq K \subseteq X$ there exists $\varepsilon > 0$ and $\emptyset \neq K'\subseteq X$ finite such that 
 $$\overline{A(U_K)} \supseteq \varepsilon U_{K'}.$$
\end{lemma}
\begin{proof}
 The bipolar theorem implies
\begin{align*}
 \overline{AU_K} &= \{f \in \Gamma(X;F) \mid |(\varphi,f)| \leq 1 \text{ for all } \varphi \in A(U_K)^\circ\}\\
 &= \{f \in \Gamma(X;F) \mid |(\varphi,f)| \leq 1 \text{ for all } \varphi \text{ s.t. } |(A'\varphi,g)|\leq 1 \text{ for all }g \in U_K\}.
\end{align*}
By the definition of the dual pairing between $\Gamma_c(X;F')$ and $\Gamma(X;F)$ we have
$$|(A'\varphi,g)|\leq 1 \text{  for all }g \in U_K$$
if and only if
$${\rm supp}\, A' \varphi \subseteq K \text{ and } \max_{x \in K} \|A'\varphi(x)\|_{F_x'} \leq 1.$$
Here, $\|\psi\|_{F_x'}$ denotes the operator norm of the functional $\psi \in F_x'$.

 We consider the vector space 
 $$V:= \{\varphi \in \Gamma_c(X;F') \mid {\rm supp} \, A'\varphi \subseteq K\}.$$
 Since $A'$ is injective, $V$ is finite-dimensional and $A'|_V\colon V \to A'(V)$ is a vector space isomorphism. In particular, there exists a finite $K'\subseteq X$ such that $\supp \varphi \subseteq K'$ for all $\varphi \in V$ (e.g. choose a finite basis $(\varphi _i)_{i \in I}$ of $V$ and set $K' = \cup_i {\rm supp}\, \varphi_i$). 
 
 We equip the vector space $\Gamma_c(X;F')$ (and all of its subspaces) with the norm $\|\cdot\|_d$ defined by 
 $$\|\varphi \|_d := \max_{x \in X} \|\varphi(x)\|_{F_x'}.$$
 Since linear operators on finite dimensional normed spaces are always continuous, the inverse  $(A'|_V)^{-1}:A'(V) \to V$ is continuous with respect to this norm. Hence, there is some constant $C > 0$ such that the $\|\cdot\|_d$ norm of every element in the set
 $$\{\varphi \in \Gamma_c(X;F') \mid {\rm supp} \, A'\varphi \subseteq K \text{ and } \|\varphi\|_d \leq 1\} = \{\varphi \in V \mid   \|A'\varphi \|_d\leq 1\}$$
 is bounded  by $C$. 
 
 The discussion at the beginning of the proof shows
 $$\overline{AU_K} =  \{f \in \Gamma(X;F) \mid |(\varphi,f)| \leq 1 \text{ for all } \varphi \in V \text{ with } \|A'\varphi\|_d \leq 1\}.$$
 Since $\|\varphi\|_d \leq C$ and ${\rm supp}\, \varphi \subseteq K'$ whenever $\varphi \in V$ and  $\|A'\varphi\|_d \leq 1$, we obtain that $f \in \Gamma(X;F)$ with $p_{K'}(f) \leq C^{-1}$ satisfies $f \in \overline{AU_K}$. This finishes the proof.
\end{proof}
\begin{remark}
 This lemma is the only place in the proof of Theorem~\ref{theorem:main} where we used the concrete structure of the space $\Gamma(X;F)$. It is the main step in the proof of Theorem~\ref{theorem:main}.
\end{remark}

As already mentioned above, we can remove the closure in the previous lemma. This is based on a standard argument for mappings between complete metrizable topological vector spaces. We include a proof for the  convenience of the reader, see also \cite[Lemma~3.9]{MV}. 
\begin{lemma}
  Let $A'$ be injective. For every finite $\emptyset \neq K \subseteq X$ there exists $\varepsilon > 0$ and $\emptyset \neq K'\subseteq X$ finite such that 
 $$ A(U_K) \supseteq \varepsilon U_{K'}.$$
\end{lemma}
\begin{proof}
 Let $\emptyset \neq K \subseteq X$ finite and let $\varepsilon' > 0$ and $\emptyset \neq K' \subseteq X$  finite  such that
 $$\overline{A(U_K)} \supseteq   \varepsilon' U_{K'}.$$
 We choose an increasing sequence of finite subsets $(K_n)$ of $X$ with $K_1 = K$ and $\cup_n K_n = X$. Lemma \ref{lemma:inclusion} yields that there exists an increasing sequence of finite sets $(K'_n)$ with $K'_1 = K'$, $\cup_n K_n' = X$ and $L_n \geq 1$ such that 
 $$\overline{A(L_n U_{K_n})} \supseteq U_{K_{n}'}.$$
 Rescaling this inclusion shows that  for every $n \in \N$, every $\gamma,\delta > 0$  and every $h \in \Gamma(X;F)$ there exists $g \in \Gamma(X;F)$ with 
 \begin{align}\label{equation:choice}
  p_{K_n}(g) \leq L_n(p_{K'_n}(h)+\gamma)\quad \text{ and }\quad p_{K'_{n+1}}(Ag - h) \leq \delta. \tag{$\heartsuit$}
 \end{align}
 Note that for $h$ with $p_{K'_n}(h) \neq 0$ the $\gamma$ in the first inequality can be omitted. Let $f \in 1/2U_{K'}$. We construct $g \in 2L_1 U_K$ with $Ag = f$, from which the claim follows with $\varepsilon = 1/4L_1$.
 
 By applying \eqref{equation:choice} to $n =1$, $h = f \in 1/2U_{K'}, \gamma = 1/2$ and $\delta = L_1/(4L_2)$ we choose $g_1$ with $p_{K_1}(g_1) \leq L_1$ and $p_{K'_2}(Ag_1 - f) \leq \frac{L_1}{4L_2}$ and  construct $(g_n)$ inductively as follows. Suppose that we have chosen $g_1,\ldots,g_n$ with
 $$p_{K_n}(g_n) \leq \frac{L_1}{2^{n-1}}\quad \text{ and } \quad p_{K'_{n+1}}(f - A(g_1 + \ldots + g_n)) \leq \frac{L_1}{2^{n+1} L_{n+1}}. $$
 By applying \eqref{equation:choice} to $n+1$,  $h = f - A(g_1 + \ldots + g_n)$, $\gamma = L_1/(2^{n+1}L_{n+1}), \delta = L_1/(2^{n+2} L_{n+2})$ we  choose $g_{n+1}$ with 
 $$p_{K_{n+1}}(g_{n+1}) \leq \frac{L_1}{2^{n}}\quad \text{ and } \quad p_{K'_{n+2}}(f - A(g_1 + \ldots + g_n) - A g_{n+1}) \leq \frac{L_1}{2^{n+2} L_{n+2}}.$$
We let $s_n := g_1 + \ldots + g_n$.  Since $(K_{n})$ is increasing and covers $X$, it follows from this construction that for every $x \in X$ the sequence $(s_n(x))$ is Cauchy in $F_x$. Thus, $(s_n)$ converges in $\Gamma(X;F)$ to some $g$. We obtain
$$p_K(g)=  \lim_{n \to \infty} p_K(s_n) \leq \liminf_{n\to \infty} \sum_{l = 1}^n p_{K_l}(g_l) \leq 2L_1. $$
Moreover, since $(K'_{n+1})$ is increasing and covers $X$, for $x \in X$ we conclude
$$\|f(x) - Ag(x) \|_x = \lim_{n \to \infty} \|f(x) - As_n(x) \|_x \leq \liminf_{n \to \infty} p_{K_{n+1}'}(f - As_n) \leq  \liminf_{n \to \infty}  \frac{L_1}{2^{n+1}} = 0.$$
Note that for the last inequality we used $L_{n+1} \geq 1$. This finishes the proof.
\end{proof}

\begin{proof}[Proof of Theorem \ref{theorem:main}]
 With the help of the previous lemma, the proof can be given exactly as sketched at the beginning of this section.
\end{proof}

\begin{remark}
 As already discussed in the introduction  there are several ways to prove Theorem~\ref{theorem:main}. Yet another functional analytic proof that is not based on Eidelheit's theorem, but uses further nontrivial results about the Fr\'echet space $\Gamma(X;F)$, was communicated to us by Jürgen Voigt. Using the structure of $\Gamma(X;F)$ and its dual it is possible to deduce from the closed range theorem \cite[Theorem~9.6.3]{Jar}   that all continuous operators on $\Gamma(X;F)$ have closed range. Since injectivity of $A'$ yields that the range of $A$ is dense, we also obtain that $A$ is surjective.
\end{remark}

\section{An application to magnetic Schrödinger operators}\label{section:applications}
In this section we apply the general surjectivity criterion to magnetic Schrödinger operators. We prove Theorem~\ref{thm:graph} and we discuss how the maximum principle from \cite{Kal} is related to our condition on the nonnegativity of $q_{V}$ in the scalar case. The strategy for the proof of Theorem~\ref{thm:graph} is as follows. Using Theorem~\ref{theorem:main} and the structure of magnetic Schrödinger operators we show that it suffices to verify that $\MM|_{\Gamma_c(X;F)}$ is injective. We establish this injectivity for the graph Laplacian $\mathcal{L}_\Wm$  respectively $\LL_{-\Wm - 2{\rm deg}}$ and then extend it to $\MM$ by domination.

 Let $F$ be a Hermitian vector bundle over $X$. We write $|\cdot|_x$ for the norm on $F_x$ which is induced by the scalar product $\as{\cdot,\cdot}_x$ (which is assumed to be linear in the second variable). In this case we identify $\Gamma(X;F)'$ and $\Gamma_c(X;F)$. The dual pairing between $\Gamma(X;F)$ and $\Gamma_c(X;F)$ is then given by
 $$(\varphi,f) = \sum_{x \in X} \as{\varphi(x),f(x)}_x,\quad f \in \Gamma(X;F), \varphi \in \Gamma_c(X;F).$$
  For $f \in \Gamma(X;F)$ we denote by $|f|$ the  function $|f|\colon X \to \R, |f|(x) = |f(x)|_x$. In what follows $b$ is a graph over $X$, $\Phi$ is a unitary connection on $F$ and $W$ is a self-adjoint bundle endomorphism and we let $\MM = \MM_{b,\Phi,W}$ the associated magnetic Schrödinger operator.
%

\begin{lemma}[Domination]\label{lemma:domination}
 For every $\varphi \in \Gamma_c(X;F)$ we have
 $$(\varphi,\mathcal M \varphi) \geq q_{\Wm}(|\varphi|).$$
\end{lemma}
\begin{proof}
 The statement follows from of a discrete Version of Kato's inequality
 $$(\varphi,\mathcal M \varphi) \geq (|\varphi|,\mathcal L_{\Wm} |\varphi|),$$
 see e.g. \cite[Lemma~2.2]{Sch2}, combined with Green's formula
 $$(|\varphi|,\mathcal L_{\Wm} |\varphi|) = q_\Wm(|\varphi|),$$
 see e.g. \cite[Lemma~2.1]{Sch2}. This finishes the proof.
\end{proof}

\begin{lemma}[Dual operator]\label{lemma:dual operator} Let $(X,b)$ be locally finite. Then $\mathcal M$ is continuous and the dual operator of $\mathcal M$ is given by $\mathcal M|_{\Gamma_c(X;F)}$. In particular, $\mathcal M$ is surjective if and only if $\mathcal M|_{\Gamma_c(X;F)}$ is injective.
\end{lemma}
\begin{proof}
 Green's formula, see e.g. \cite[Lemma~2.1]{Sch2}, implies that for $\varphi,\psi \in \Gamma_c(X;F)$ we have
$$(\varphi,\MM \psi) = \sum_{x \in X} \as{\MM \varphi(x),\psi (x)}_x. $$
The local finiteness of $(X,b)$ yields that $\MM$ is continuous (it has finite hopping range) and $\MM \Gamma_c(X;F) \subseteq \Gamma_c(X;F)$. Hence, the above identity shows $(\varphi,\MM \psi) = (\MM \varphi,\psi)$. 

With these   observations the 'In particular'-part follows  from Theorem~\ref{theorem:main}.
\end{proof}

\begin{remark}
  In view of the identification of  $\Gamma_c(X;F')$ with $\Gamma_c(X;F)$, a subspace of $\Gamma(X;F)$ we call a continuous  operator $A\colon\Gamma(X;F) \to \Gamma(X;F)$ {\em symmetric} if its dual operator satisfies $A' = A|_{\Gamma_c(X;F)}$. In the previous lemma we proved that any magnetic Schrödinger operator is symmetric  and it is not hard to prove that any continuous symmetric operator is indeed a magnetic Schrödinger operator. 
\end{remark}

\begin{lemma}[Kernel of $q_V$]\label{lemma:kernel qv}
 Let $V\colon X \to \R$. If $q_V$ is nonnegative, then for every $h \in C_c(X)$ with $q_V(h) = 0$ the set $\{h \neq 0\}$ is a union of connected components. In particular, if all connected components of $(X,b)$ are infinite, every such $h$ has to vanish.
\end{lemma}
\begin{proof}  As can easily be seen from the definition of the form $q_V$, for any $\varphi \in C_c(X)$ we have $q_V(|\varphi|) \leq q_V(\varphi)$.  Since $0 \leq q_V(|h|) \leq q_V(h) = 0$, it suffices to consider the case $h \geq 0$.  As a first step we show that for any $\varphi \in C_c(X)$ the inequality
\begin{align}\label{inequality:important}
 q_V( \varphi 1_{\{h > 0\}}) \leq q_V(\varphi)\tag{$\heartsuit$}
\end{align}
holds.

Since $\varphi$ has finite support, there exists a constant $M > 0$ such that 
$$ \varphi 1_{\{h > 0\}} = (\varphi \wedge (Mh)) \vee (-Mh).$$
Here $\wedge$ denotes the pointwise minimum of two functions and $\vee$ denotes their pointwise maximum. Using the formula $2 (a \wedge b) = a + b - |a-b|$, $a,b \in \R$, that $q_V$ is a nonnegative quadratic form and $q_V(|\psi|) \leq q_V(\psi)$, $\psi \in C_c(X)$, we obtain
\begin{align*}
 2 q_V(\varphi \wedge (Mh))^{1/2} &\leq  q_V(\varphi + Mh)^{1/2} + q_V(|\varphi - Mh|)^{1/2}\\
 &\leq 2 q_V(\varphi)^{1/2} + 2 q_V(Mh)^{1/2}\\
 &= 2 q_V(\varphi)^{1/2}.  
\end{align*}
This inequality combined with a similar argument for $\varphi \vee (-Mh)$ implies Inequality~\eqref{inequality:important}.
 
 To finish the proof we have to show that for a given  $x \in X$ with $h(x) >  0$ and a given $y \in X$ that is  connected with $x$ we have $h(y) > 0$. By induction we can assume that $y$ is a neighbor of $x$, that is, $b(x,y) > 0$. Suppose that $ h(y) = 0$. We write $\delta_z$ for the function on $X$ that is $1$ at $z$ and $0$ otherwise.  For $\alpha >0$ the definition of $q_V$, $h(y) = 0$ and Inequality~\eqref{inequality:important} yield
 $$q_V(\delta_x) = q_V( (\delta_x + \alpha \delta_y) 1_{\{h > 0\}}) \leq q_V(\delta_x + \alpha \delta_y) = q_V(\delta_x) + \alpha^2 q_V(\delta_y) - 2\alpha b(x,y).$$
 Rearranging and dividing by $\alpha > 0$ shows $2 b(x,y) \leq \alpha q_V(\delta_y)$. Letting $\alpha \to  0+$ yields $b(x,y) = 0$, a contradiction.
\end{proof}

\begin{remark}
  The idea for the  proof  of Inequality~\eqref{inequality:important} and its application in the proof of the previous theorem are taken from \cite{Sch1}. 
\end{remark}

\begin{proof}[Proof of Theorem~\ref{thm:graph}]  Obviously we have
$$- \MM_{b,\Phi,W} =  \MM_{b,-\Phi,-W - 2 {\rm deg}}$$
and $(-W - 2 {\rm deg})_{\rm min}(x) = - \Ws(x) - 2{\rm deg}(x)$. Moreover, $- \MM_{b,\Phi,W}$ is surjective if and only if $\MM_{b,\Phi,W}$ is surjective. Therefore, it suffices to consider the case that $q_{\Wm}$ is nonnegative.

Since $(X,b)$ is assumed to be locally finite,  by Lemma~\ref{lemma:dual operator} it suffices to verify that $\MM|_{\Gamma_c(X;F)}$ is injective. Let $\varphi \in \Gamma_c(X;F)$ with $\MM \varphi = 0$. Lemma~\ref{lemma:domination} and that $q_{\Wm}$ is nonnegative show
 $$0 \leq q_{\Wm}(|\varphi|) \leq (\varphi,\MM \varphi) = 0.$$
 Since the connected components of $(X,b)$ are assumed to be infinite, we infer $|\varphi| = 0$ from Lemma~\ref{lemma:kernel qv} and hence $\varphi = 0$. This finishes the proof. 
\end{proof}

We finish this section with the discussion of a pointwise maximum principle in the scalar case. For $x \in X$ and $n \in \N_0$ we denote by $B_n(x)$ the vertices with combinatorial distance less or equal than $n$ from $x$. As in \cite{Kal} we say that a continuous operator $A\colon C(X)\to C(X)$ satisfies the {\em maximum principle at $x \in X$} if there exists $n \in \N$ such that for every $f \in C(X)$ the identities $Af(x) = 0$ and
$$|f(x)| = \sup_{y \in B_n(x)} |f(y)|$$
 imply $|f(x)| = |f(y)|$ for all $y \in B_n(x)$. Moreover, $A$ is said to satisfy the  {\em pointwise maximum principle} if it satisfies the  maximum principle at every $x \in X$.  We obtain the following characterization of the maximum principle for $\LL_V$.

\begin{proposition}\label{proposition:maximum principle}
  Let $(X,b)$ be locally finite and let $x \in X$. The following assertions are equivalent:
  \begin{enumerate}[(i)]
   \item $\LL_V$ satisfies the maximum principle at $x$.
   \item $V(x) \geq 0$ or $V(x) + 2 {\rm deg}(x) \leq 0$.
  \end{enumerate}
 \end{proposition}
 
 Before proving this proposition we note the following elementary observation. 
 
 \begin{lemma}\label{lemma:maximum principle}
  Let $(X,b)$ be locally finite, let $f \in C(X)$ nonnegative and  let $x \in X$. If $V(x) \geq 0$, then  $\LL_V f(x)  \leq 0$ and 
  $$f(x) = \sup_{y \in B_1(x)} f(y)$$
  imply $f(x) = f(y)$ for all $y \in B_1(x)$.
 \end{lemma}
 \begin{proof}
 The nonnegativity of $f$ and $V(x) \geq 0$ yield
  $$0 \geq \LL_V f(x) = \sum_{y \in X} b(x,y)(f(x)- f(y)) + V(x) f(x) \geq \sum_{y \in X} b(x,y)(f(x)- f(y)).$$
 From this inequality the claim follows immediately.
 \end{proof}
 \begin{proof}[Proof of Proposition~\ref{proposition:maximum principle}]
  (ii) $\Rightarrow$ (i): If $V(x) \geq 0$ the statement follows directly from the previous lemma as $\LL_V f = 0$ implies $\LL_V|f| \leq 0$. For the case $V(x) + 2 {\rm deg}(x) \leq 0$ we argue as follows. Let $f \in C(X)$ with $\LL_V f(x) = 0$ and $|f(x)| = \sup_{y \in B_1(x)} |f(y)|$ be given. Without loss of generality we assume $f(x) \geq 0$. The properties of $f$ yield
  \begin{align*}
  0 =  - \LL_V f(x) &= \sum_{y \in X} b(x,y) (f(y) - f(x)) - V(x) f(x)\\
   &= \sum_{y \in X} b(x,y) (f(y) + f(x)) - (V(x) + 2 {\rm deg}(x)) f(x)\\
   &\geq  \sum_{y \in X} b(x,y) (|f(x)| - |f(y)|) - (V(x) + 2 {\rm deg}(x)) |f(x)|\\
   &= \LL_{-(V + 2{\rm deg})} |f|(x).
  \end{align*}
  With this at hand, the statement follows from Lemma~\ref{lemma:maximum principle}.

%
%
  
  (i) $\Rightarrow$ (ii): Suppose that (ii) does not hold. Then $V(x) < 0 < V(x) + 2 {\rm deg}(x)$.  For $\beta \in \R$ we consider $f_\beta \in C(X)$ defined by 
 $$f_\beta(y) =  \begin{cases}1  &\text{if } y = x\\
 \beta   &\text{if } y \sim x\\
 0 &\text{else}
   \end{cases}.
$$
 It satisfies
 $$\LL_V f_\beta(x) = (1 -  \beta) {\rm deg} (x)  +  V(x).  $$
 In particular, for $\beta =  (V(x) + {\rm deg}(x))/{\rm deg}(x)$ we have $\LL_V f_\beta(x) = 0$. The bounds on $V(x)$ yield $|\beta| < 1$, so that $|f_\beta|$ attains its global maximum at $x$. Hence, $f_\beta$ violates the maximum principle at $x$. 
 \end{proof}

 With the characterization of the maximum principle we also recover the main result from \cite{Kal}  for the Schrödinger operator $\LL_V$ and make it somewhat more explicit.
 
\begin{corollary}
 If $(X,b)$ is locally finite, all connected components are infinite and for all $x \in X$ either $V(x) \geq 0$ or $V(x) + 2{\rm deg}(x) \leq 0$, then $\LL_V$ is surjective. 
\end{corollary}
\begin{proof}
 The pointwise maximum principle for $\LL_V$ and that connected components of $(X,b)$ are infinite clearly imply that any $f \in C_c(X)$ with $\LL_V f = 0$ must vanish. Hence, we can deduce the surjectivity of $\LL_V$ as in the proof of Theorem~\ref{thm:graph}.
\end{proof}
\begin{remark}
\begin{enumerate}[(a)]
\item The main result of \cite{Kal} states that a continuous linear operator $A$ on $C(X)$ is surjective if it satisfies the pointwise maximum principle (with respect to a locally finite connected infinite graph). Hence, the previous corollary is the main result of \cite{Kal}  applied to the Schrödinger operator $\LL_V$. For other continuous operators on $C(X)$ the  pointwise maximum principle cannot be easily characterized.
 \item  As remarked above after Theorem~\ref{thm:graph}, our form criterion in Theorem~\ref{thm:graph} can also treat certain $V$ in the regime $V + 2 {\rm deg} > 0 >V$, which are not covered by this corollary. However, note that our form criterion does not cover the case when $V(x) \geq 0$ for some $x \in X$ and $V(x) + 2 {\rm deg}(x) \leq 0$ for some $x \in X$.  
\end{enumerate}

\end{remark}

\section{Examples}\label{section:examples}
 
 In this section we  illustrate with several examples that none of the assumptions of Theorem~\ref{thm:graph} can be dropped for inferring surjectivity. 
 
 The necessity of infinite connected components can be seen as follows. If $X$ is finite (or $(X,b)$ is locally finite and has at least one finite connected component),  the constant functions (the functions that are constant on finite connected components and vanish elsewhere) are finitely supported  eigenfunctions to the eigenvalue $0$ for the graph Laplacian $\LL = \LL_0$. Hence, on such graphs the Laplacian $\LL$ can not be surjective.

 The following example shows that local finiteness is also essential for surjectivity. 
 \begin{example}[Infinite Star]
  Let $X= \N_0$ and let $(b_n)_{n \in \N}$ be a sequence of nonnegative numbers with $\sum_{n} b_n < \infty$. We define $b\colon\N_0 \times \N_0 \to [0,\infty)$  by $b(n,0) = b(0,n) = b_n$, $n \in \N$, and $b(n,m) = 0$, otherwise. In this case, we have $D(\LL) = \{f \in C(\N_0) \mid \sum_n b_n |f(n)| < \infty\}$.
  
  Now suppose that $g \in D(\LL)$ and $f \in C(X)$ satisfy $\LL g = f$. For $n \in \N$ we obtain
  $$f(n) = \LL g (n) =  b_n (g(n) - g(0)). $$
  Since  $g \in D(\LL)$, this implies $f \in \ell^1(\N_0)$. Moreover, for $n = 0$ we get
  $$f(0) = \LL g (0) = \sum_{n \in \N} b_n(g(0) - g(n)). $$
  Substituting the first identity into the second yields that
  $$f(0) = - \sum_{n \in\N} f(n)$$
  is a necessary condition for $f$ to lie in the image of $\LL$. It is readily verified that this condition is also sufficient, that is, 
  $$\LL C(\N_0) = \{ f \in \ell^1(\N_0) \mid f(0) = -  \sum_{n \in\N} f(n)\}.$$

 \end{example}
 
 The last example in this section shows that the nonnegativity of $q_{\Wm}$ cannot be dropped in Theorem~\ref{thm:graph}. The reason behind this is the existence of finitely supported eigenfunctions on certain locally finite graphs. If $\varphi \in C_c(X)$ satisfies $\LL\varphi = \lambda \varphi$ for some $\lambda > 0$, then $\LL - \lambda = \LL_{-\lambda}$ is not surjective by Lemma~\ref{lemma:dual operator}.

 \begin{example} 
  According to the discussion preceding this example, we construct a  locally finite graph with infinite connected components that admits finitely supported eigenfunctions for the Laplacian $\LL = \LL_0$.  We first consider the following finite graph with standard weights.

 
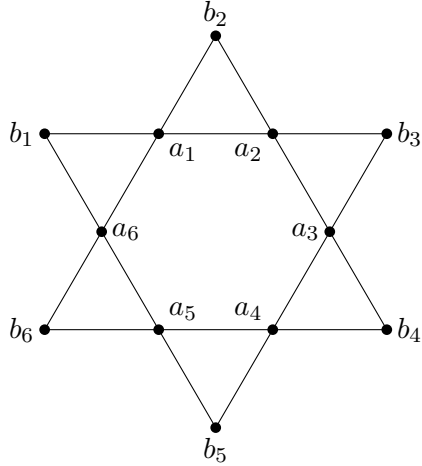
\begin{center}
 \begin{tikzpicture}[scale=0.5]
\coordinate[label=left:$b_1$] (A) at (0,0);
\coordinate[label=right:$b_3$] (B) at (9,0);
\coordinate[label=below:$b_5$] (C) at (4.5,{-sqrt(243)/2});
\coordinate[label=left:$b_6$] (D) at (0,{-sqrt(27)});
\coordinate[label=right:$b_4$] (E) at (9,{-sqrt(27)});
\coordinate[label=above:$b_2$] (F) at (4.5,{sqrt(27)/2});
\coordinate[label=below right:$a_1$] (a) at ($(A)!1/3!(B)$);
\coordinate[label=below left:$a_2$] (b) at ($(A)!2/3!(B)$);
\coordinate[label=left:$a_3$] (c) at ($(F)!2/3!(E)$);
\coordinate[label=above left:$a_4$] (d) at ($(B)!2/3!(C)$);
\coordinate[label=above right:$a_5$] (e) at ($(E)!2/3!(D)$);
\coordinate[label=right:$a_6$] (f) at ($(D)!1/3!(F)$);

\draw (A) -- (B);
\draw (B) -- (C);
\draw (C) -- (A);
\draw (D) -- (E);
\draw (D) -- (F);
\draw (E) -- (F);

\fill (A) circle (4pt);
\fill (B) circle (4pt);
\fill (C) circle (4pt);
\fill (D) circle (4pt);
\fill (E) circle (4pt);
\fill (F) circle (4pt);
\fill (a) circle (4pt);
\fill (b) circle (4pt);
\fill (c) circle (4pt);
\fill (d) circle (4pt);
\fill (e) circle (4pt);
\fill (f) circle (4pt);

 \end{tikzpicture}
 \end{center}
 
More precisely, we let  $X_1  = \{a_1,\ldots,a_6,b_1,\ldots, b_6\}$  and define $b_1\colon X_1 \times X_1 \to \{0,1\}$ by $b_1(x,y) = 1$ if there is an edge   between $x$ and $y$ in the picture above, and $b_1(x,y) = 0$, otherwise. The function $\varphi\colon X_1 \to \R$ given by $\varphi(b_i) = 0$ and $\varphi(a_i) = (-1)^i$, $i = 1,\ldots,6,$ satisfies $\LL^1 \varphi  = 6 \varphi$, where $\LL^1$ is the Laplacian of the weighted graph $(X_1,b_1)$. This is an eigenfunction that is supported on the 'interior' but not on the 'boundary' of $(X_1,b_1)$.   

Next we glue an infinite graph to the vertex $b_1$. We let  $(X_2,b_2)$ be a  locally finite graph with infinite connected components and choose $o \in X_2$. We define $X: = X_1 \sqcup X_2$ and define $b\colon X \times X \to [0,\infty)$  through $b|_{X_1 \times X_1} = b_1, b|_{X_2 \times X_2} = b_2, b(o,b_1) = b(b_1,o) = 1$ and $b(x,y) = 0$, otherwise. 

The so constructed graph $(X,b)$ is locally finite and has infinite connected components. If we denote by $\psi$ the function with $\psi = \varphi$ on $X_1$ and $\varphi = 0$ on $X_2$, it is readily verified that $\LL \psi = 6 \psi$, where $\LL$ is the Laplacian of $(X,b)$.
 \end{example}

 \begin{remark}
  Note that nonnegativity of $q_{V}$ or $q_{-V - 2 {\rm deg}}$  is not necessary for the surjectivity of $\LL_V$. There are infinite connected locally finite planar graphs whose Laplacians do not admit finitely supported eigenfunctions, see \cite{Kel}. In this case, for each $\lambda \in \R$ the restriction of the operator $\LL - \lambda = \LL_{-\lambda}$ to $C_c(X)$ is injective and hence $\LL_{-\lambda}$ is surjective.  Moreover, such $\lambda$ can always be chosen such that neither the form $q_{-\lambda}$ nor the form $q_{\lambda -2 {\rm deg}}$ are nonnegative. For example, $\lambda$ in the spectrum of the restriction of $\LL$ to $\ell^2(X)$ that do not lie at the edges of the spectrum have this property. We refrain from giving details. 
 \end{remark}

 \bibliographystyle{plain}
 
\bibliography{literatur}
\end{document}